\documentclass{article}
\usepackage{caption}
\usepackage{graphicx} 

\usepackage{amsthm,amsfonts,amssymb,amsmath,epsf, verbatim}
\usepackage{cite}
\usepackage{tikz}
\usepackage{url}
\usepackage[verbose]{placeins}
\usetikzlibrary{graphs}
\usetikzlibrary{graphs.standard}
\usetikzlibrary{positioning, quotes}
\usetikzlibrary{calc}
\tikzgraphsset{edges={draw,semithick}, nodes={circle,draw,semithick}}
\usepackage[letterpaper,top=2cm,bottom=2cm,left=3cm,right=3cm,marginparwidth=1.75cm]{geometry}
\tikzset{
    position/.style args={#1:#2 from #3}{
        at=(#3.#1), anchor=#1+180, shift=(#1:#2)
    }
}

\newtheorem{theorem}{Theorem}

\newtheorem{proposition}[theorem]{Proposition}
\newtheorem{observation}[theorem]{Observation}
\newtheorem{conjecture}[theorem]{Conjecture}

\title{Dots and Boxes on Certain Families of Graphs}
\author{Vedant Aryan\footnote{va0200@princeton.edu }, Alana Palmer\footnote{alanapalmer13@gmail.com}, Alexander Skula\footnote{skula@mit.edu}, Matthew Woolbert\footnote{ matthew.k.woolbert@gmail.com},\\ Joshua Zelinsky\footnote{jzelinsky@hopkins.edu}}
\date{}

\begin{document}

\maketitle

\begin{abstract}
    We investigate the Dots and Boxes game,  also known as ``Strings and Coins,'' for certain specific families of graphs. These include complete graphs, wheel graphs, and friendship graphs.  
\end{abstract}

\section*{Introduction}

 The traditional Dots and Boxes game involves two players taking turns drawing line segments between adjacent dots on a fixed grid. The objective is to complete the boundaries of boxes to claim them. Whenever a player completes a box, they earn a point and must draw another line segment.
 
 This game can be represented as a graph by replacing each box with a vertex and each line segment with an edge connecting the corresponding vertices. A line on the boundary corresponds to a self-loop.

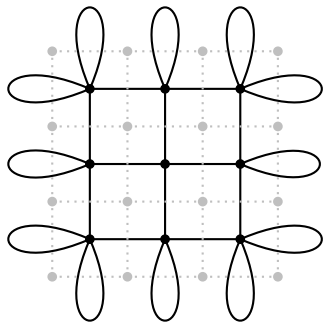
\begin{figure}[hbt]
\begin{tikzpicture}[thick, color=black, main/.style={draw, circle, inner sep=0, outer sep=0, minimum size=1mm}]
    \foreach \x in {0,1,2}
        \foreach \y in {0,1,2}
        {
            \node[main, fill=black] at (\x,\y) (\x\y) {};
        }

    \foreach \x in {0,1,2}
        \foreach \y in {0,1,2}
        {
            \ifnum\x<2
                \draw (\x,\y) to (\x+1,\y);
            \fi
            \ifnum\y<2
                \draw (\x,\y) -- (\x,\y+1);
            \fi
        }
    \draw (00) to[out=157.5, in=202.5, looseness=100] (00);
    \draw (00) to[out=247.5, in=292.5, looseness=100] (00);
    \draw (01) to[out=202.5, in=157.5, looseness=100] (01);
    \draw (02) to[out=157.5, in=202.5, looseness=100] (02);
    \draw (02) to[out=67.5, in=112.5, looseness=100] (02);
    \draw (10) to[out=-112.5, in=-67.5, looseness=100] (10);
    \draw (12) to[out=112.5, in=67.5, looseness=100] (12);
    \draw (20) to[out=247.5, in=292.5, looseness=100] (20);
    \draw (20) to[out=337.5, in=22.5, looseness=100] (20);
    \draw (21) to[out=22.5, in=-22.5, looseness=100] (21);
    \draw (22) to[out=337.5, in=22.5, looseness=100] (22);
    \draw (22) to[out=67.5, in=112.5, looseness=100] (22);
    
    \foreach \x in {-1,1,3,5}
        \foreach \y in {-1,1,3,5}
        {
            \node[main, fill=lightgray, color=lightgray] at (\x/2,\y/2) (\x,\y) {};
        }
    \foreach \x in {-1,1,3,5}
        \foreach \y in {-1,1,3,5}
        {
            \ifnum\x<5
                \draw[dotted, lightgray] (\x/2,\y/2) to (\x/2+1,\y/2);
            \fi
            \ifnum\y<5
                \draw[dotted, lightgray] (\x/2,\y/2) -- (\x/2,\y/2+1);
            \fi
        }
\end{tikzpicture}
\centering
\caption{Example of Reimplementation of Dots and Boxes}
\label{fig:reimplementation}
\end{figure}
\FloatBarrier
 
In the graph version of the game, players take turns removing edges. Whenever all the edges connected to a given vertex are removed, the player earns a point and must take another edge.

\par
This variant of Dots and Boxes has been studied since the early 2000s, with the most notable work being the book ``The Dots and Boxes Game'' 
by Elwyn Berlekamp \cite{Berlekamp}. This graph version of the game is sometimes called ``Strings and Coins,'' where each vertex is imagined as a coin and each edge as a string that needs to be cut.

When playing the game on a graph, one does not need to restrict to a graph that correspond to actual grids, although it is assumed that in the starting configuration, every vertex has a non-zero degree. Prior work has studied the computational complexity of this game with an arbitrary starting graph. In particular, Demaine and Diomidova \cite{DD} proved that determining the winner of this game is strongly PSPACE-complete.

The main tactic of their argument is to examine a closely related game where the player who makes the last move loses. This game, called Nimstring, can be embedded into the Dots and Boxes game. Given a graph $G$, one can create a new graph $H$ by taking the disjoint union of $G$ a very large cycle graph $C$ with more vertices than the rest of $G$. Whoever cuts the very large cycle first will lose. Thus, the player who moves last on $G$ is forced to cut $C$, ensuring that the other player wins the overall game. 

A similar but distinct proof that the game is PSPACE-complete was also given by Buchin, Hagedoorn, Kostitsyna, and van Mulken \cite{BHKM}.

The approach of this paper is to examine various families of graphs and determine the winner within those families, rather than considering the general game. In this context, it bears some similarity to a prior paper by Jobson, Sledd, White, and Wildstrom \cite{JSWW} who determined the winner for various subgraphs of the standard Dots and Boxes game. Among other results, they proved a conjecture by Nowakowski and Ottaway (see pg. 470 of \cite{Nowakowski}) showing that the related game of Dots and Triangles, played on a triangular grid with two rows and $n+1$ columns of vertices, is a first player win. They also proved that in Dots and Boxes with two rows of dots and $n+1$ columns, the first player can at least force a tie.

Unlike their work, this paper examines graph families that generally do not arise from a standard grid. The families investigated here include loopy stars, loopy trees, complete graphs, complete bipartite graphs, friendship graphs, wheel graphs, and hypercubes. For some of these families, we provide complete proofs of their general behavior. For others, we offer conjectures based on data from many cases. We have complete descriptions with proofs of the winners for loopy stars, loopy trees, and friendship graphs. We have conjectured behavior for complete graphs, wheel graphs, and hypercubes. In the case of wheel graphs, we also outline what we hope might be developed into a proof of the conjectured behavior by connecting them to certain other graph families.

We will refer to a graph as a ``first player win'' graph if the first player wins the corresponding game, a ``second player win'' graph if the second player wins the corresponding game, and a ``tie'' graph if the game ends in a tie.

Certain graph families are essentially trivial. If $G$ is a tree, then it is a first player win, as the first player can take all vertices. 

The relationship between the winner on similar graphs can be subtle. For example, one might conjecture that if $G$ is a graph, adding a single loop to an existing vertex will always swap the winner. However, this is not necessarily the case. A counterexample is a $C_5$ with a single loop attached to one vertex. This graph is a first player win. If a loop is added to one of the two vertices adjacent to the vertex with the existing loop, the first player can still win by taking the edge between the two vertices with loops. However, if a graph $G$ is a second player win, then adding a single loop to an existing vertex in $G$ always makes it a first player win, because the first player can take the loop and then play as the second player on the reduced graph.

\section{Preliminaries}
Berlekamp discussed that the winner on a cycle $C_n$ is always Player 2. After Player 1 makes the first move, the cycle becomes a path, allowing Player 2 to take all the remaining vertices. Player 2 also wins on disconnected unions of cycles where each cycle has a length of at least 8. After Player 1 transforms one of the cycles into a path with the first move, Player 2 continues as follows. Let $n$ be the number of vertices in the cycle that Player 1 broke with their first move. Player 2 takes vertices and edges from the path until at most 3 vertices are left. The remaining edges and vertices can then be taken to the benefit of Player 2. At this point, Player 1 must then break another large cycle, and Player 2 can repeat this strategy until Player 1 breaks the final cycle, at which point Player 2 can take all the remaining vertices and win.

The winner on any tree is always Player 1. Every tree has at least two leaves (vertices with degree one), so Player 1 can always remove the edge connecting to a leaf and take the vertex it exposes. Player 1 can continue this strategy on the remaining tree until all vertices are taken. Similarly, in forests, Player 1 can employ the same leaf-removal strategy on each tree, ultimately taking all the vertices.

\section{Computations}

We developed an algorithm that, given a graph, determines the winner and calculates the best score they can achieve. The algorithm was implemented in Python using the NetworkX library, which is designed for graph theory applications.

The program uses canonical representations of graph families such as complete graphs, wheel graphs, friendship graphs, and hypercubes. It is also capable of handling custom graph types by dynamically generating these structures based on input edges.

The software operates by taking a graph structure as input and recursively evaluating each possible move. It employs memoization to check for previously calculated outcomes of graph configurations, thereby reducing redundant computations. This approach allows for a more efficient exploration of outcomes across various graph configurations compared to a naive, brute-force method.

The software is available on our GitHub page \cite{GitHub}.

We will discuss empirical results for each graph family in their corresponding sections.

\section{Friendship Graphs}
Friendship graphs, denoted as $F_n$, are formed by $n$ copies of $K_3$ all sharing a single vertex. 

\begin{figure}[hbt!]
\centering
\begin{tikzpicture}[thick, main/.style={draw, circle, inner sep=0, outer sep=0, minimum size=1.5mm}]
    \node[main,fill=black] (1) {};
    \node[main,fill=black] (2)[position=-120:{2cm} from 1] {};
    \node[main,fill=black] (3)[position=180:{2cm} from 1] {};
    \node[main,fill=black] (4)[position=60:{2cm} from 1] {};
    \node[main,fill=black] (5)[position=-240:{2cm} from 1] {};
    \node[main,fill=black] (6)[position=0:{2cm} from 1] {};
    \node[main,fill=black] (7)[position=300:{2cm} from 1] {};
    \draw (1) -- (2) -- (3) -- (1) -- (4) -- (5) -- (1) -- (6) -- (7) -- (1);
\end{tikzpicture}
\caption{An example of a Friendship Graph $F_3$ consisting of three $K_3$s sharing a common center vertex.}
\label{fig:friendship-graph-f3}
\end{figure}
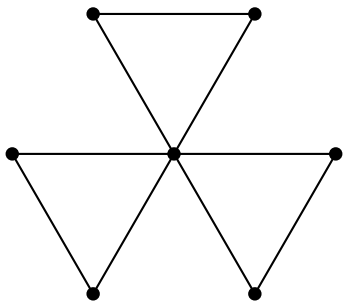
\FloatBarrier

\begin{theorem}
    In a game played on a friendship graph $F_n$, the winner alternates between Player 1 and Player 2, with Player 1 winning by one point for even $n$, and Player 2 winning by three points for odd $n$.  Specifically, for even $n$, the score is $(n +1, n)$, while for odd $n$, the score is $(n- 1, n + 2)$.
\end{theorem}

\begin{proof}
We will prove this by induction on $n$.

\textbf{Base Cases:}

\begin{itemize}
    \item $F_1$:  The graph is a single $K_3$. Player 2 wins with a score of (0, 3). This matches the theorem: $n=1$ is odd, and the score is (1-1, 1+2) = (0, 3).
    \item $F_2$: Player 1 wins. Player 1 takes an edge that is not connected to the center vertex. Player 2 must take the two adjacent vertices.  Player 1 ends with a score of (3,2), matching the theorem: $n=2$ is even and the score is (2+1, 2) = (3, 2).
\end{itemize}

\textbf{Inductive Step:}

Assume that the theorem holds for $F_k$ for all $k < n$.  We will show it holds for $F_n$. We need to consider two cases: $n$ is even, and $n$ is odd.

\textbf{Case 1: $n$ is even.} Player 1's first move is to take an edge of one of the $K_3$s that is not connected to the central vertex.  Player 2 must then take the two vertices adjacent to that edge (otherwise, Player 1 will simply take them on their next turn). This leaves the remaining graph as $F_{n-1}$, and Player 2 has been forced to start the game on $F_{n-1}$. By the inductive hypothesis, since $n-1$ is odd, the score on $F_{n-1}$ will be ($(n-1) - 1$, $(n-1) + 2$) = ($n-2$, $n+1$). Thus, the overall score for $F_n$ will be ($2 + (n-2)$, $n+1$) = ($n$, $n+1$).  Since Player 1 took the first edge as well, the overall score is $(n+1, n)$. This matches the formula for even $n$. Therefore, Player 1 wins by one point.

\textbf{Case 2: $n$ is odd.} Player 1's first move is to take any edge. This forces Player 2 to play $F_{n-1}$.  By the inductive hypothesis, since $n-1$ is even, the score on $F_{n-1}$ will be  (($n-1$) + 1, $n-1$) = ($n$, $n-1$).  Thus, the overall score for $F_n$ will be ($n, n-1$). Since Player 1 took the first edge, the overall score is either ($n+1, n-1$) or ($n, n$), depending on whether the edge came with a vertex. However, Player 1 could have, instead, taken a loop from one of the end vertices, for a final score of ($n-1, n+2$). Therefore, Player 2 wins by three points.

Therefore, by induction, the theorem holds for all $n \ge 1$.
\end{proof}

\begin{table}[hbt!]
\centering
\begin{tabular}{|c|c|c|}
\hline
\textbf{\# of Triangles} & \textbf{Winner} & \textbf{Score} \\
\hline
1 & P2 & (0, 3) \\
2 & P1 & (3, 2) \\
3 & P2 & (2, 5) \\
4 & P1 & (5, 4) \\
5 & P2 & (4, 7) \\
6 & P1 & (7, 6) \\
7 & P2 & (6, 9) \\
8 & P1 & (9, 8) \\
\hline
\end{tabular}
\caption{Computed outcomes for Friendship Graphs based on the number of triangles.}
\label{table:friendship-graph-outcomes}
\end{table}
\FloatBarrier 

\section{Pinwheel Graphs}
A Pinwheel Graph $PW_n$ is formed by joining $n$ copies of $C_4$ to a single shared vertex.
\begin{theorem}
    In the Pinwheel graph game, Player 2 wins for all $PW_n$.  Specifically, the score is (0, 4) for $n = 1$, (2, 5) for $n = 2$, and for $n \geq 3$, we have a score of $(n + 2, n + 3)$ for even $n$, and $(n + 2, n + 4)$ for odd $n$.
\end{theorem}
\begin{proof}
    We will prove this theorem by examining the winning strategies for Player 2 in various cases of $PW_n$. The general strategy for Player 2 will revolve around strategically responding to Player 1's moves to either immediately capture an entire square or force Player 1 into a situation where they must concede vertices to Player 2.  A key element will be to control the center vertex as long as possible and to ensure that if Player 1 takes the center vertex, they do not gain an overwhelming advantage.
    
    \textbf{Base Cases:}

    \begin{itemize}
        \item $PW_1$: is isomorphic to $C_4$, which is known to be a Player 2 win with a score of (0, 4).
        \item $PW_2$: Player 2 wins with a score of (2, 5). Suppose Player 1's first move is to remove the center edge of one cycle. Player 2 immediately finishes the cycle to gain 4 points before Player 1 can claim any points.  In this situation Player 2 wins with a score of (2, 5). Suppose Player 1's first move is to remove an outer edge of one cycle. Player 2 then takes the neighbor vertex to the edge. From that vertex, Player 2 takes the edge that would finish the adjacent cycle, leaving Player 1 to take all 2 vertices, and player 2 to take all 5 of the other vertices. Thus, Player 2 wins with a score of (2, 5). 
        \item $PW_3$: Player 2 wins by two points. Player 2 can use the strategy of responding to Player 1's moves to either immediately capture an entire square or force Player 1 into a situation where they must concede vertices to Player 2, leading to Player 2 winning.
        \item $PW_4$: Player 2 wins by one point. Player 2 can either implement the strategy of responding to Player 1's moves to either immediately capture an entire square or force Player 1 into a situation where they must concede vertices to Player 2 or use the strategy of taking the first cycle to force Player 1 into $PW_3$ with 3 vertices to make up 2 points.
    \end{itemize}
    
    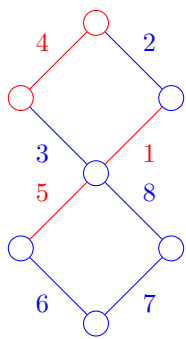
\begin{figure}[hbt!]
\centering
\begin{tikzpicture}
    \node[circle, draw, color=blue] (1) at (0,0) {};
    \node[circle, draw, color=blue] (2) at (1,1) {};
    \node[circle, draw, color=red] (3) at (-1,1) {};
    \node[circle, draw, color=red] (4) at (0,2) {};
    \node[circle, draw, color=blue] (5) at (-1,-1) {};
    \node[circle, draw, color=blue] (6) at (1,-1) {};
    \node[circle, draw, color=blue] (7) at (0,-2) {};  
    \draw[red] (2) edge["{1}"] (1);
    \draw[blue] (4) edge["{2}"] (2);
    \draw[red] (3) edge["{4}"] (4);
    \draw[blue] (1) edge["{3}"] (3);
    \draw[red] (5) edge["{5}"] (1);
    \draw[blue] (7) edge["{6}"] (5);
    \draw[blue] (6) edge["{7}"] (7);
    \draw[blue] (1) edge["{8}"] (6);
\end{tikzpicture}
\caption{Winning strategy for Player 2 in $PW_2$}
\label{fig:pw2}
\end{figure}
\FloatBarrier 

\textbf{Inductive Step:} Assume that Player 2 wins for all $PW_k$ where $k < n$ and $n \geq 5$. We want to show that Player 2 wins for $PW_n$.

Suppose Player 1's first move is to take the center edge. Player 2 can then take the remaining edges in the same cycle and become Player 1 on a smaller Pinwheel graph and with extra verticies, guaranteeing a win. Suppose Player 1's first move is to take an edge in the outer cycle. Player 2 can then move on from there to take an edge on the same cycle or a different one, depending on what leads to a greater advantage.

Therefore, Player 2 can win on all Pinwheel graphs.
\end{proof}

\begin{table}[hbt!]
\centering
\begin{tabular}{|c|c|c|}
\hline
\textbf{\# of Squares} & \textbf{Winner} & \textbf{Score} \\
\hline
1 & P2 & (0, 4) \\
2 & P2 & (2, 5) \\
3 & P2 & (4, 6) \\
4 & P2 & (6, 7) \\
5 & P2 & (7, 9) \\
6 & P2 & (9, 10) \\
7 & P2 & (10, 12) \\
8 & P2 & (12, 13) \\
\hline
\end{tabular}
\caption{Computed outcomes for Pinwheel Graphs based on the number of squares.}
\label{table:pinwheel-graph-outcomes}
\end{table}
\FloatBarrier 

\section{Loopy Stars}

In this section, we introduce a unified notation for loopy star graphs. A loopy star graph is denoted by $L(n, k, l)$, where:

\begin{enumerate}
    \item $n$ = number of branches
    \item $k$ = number of loops at the leaf of each branch
    \item $l$ = length (number of edges) of the branches
\end{enumerate}

If a variable is 1, it can be left out from the right. So, $L(n)$ is a loopy star with $n$ branches of length 1 with 1 loop each, and $L(n, k)$ is a loopy star with $n$ branches of length 1, each with $k$ loops.

Loopy stars are a class of graphs characterized by a central vertex from which multiple branches emanate, each terminating in a loop. They are called loopy stars because they are star graphs with loops attached to each branch.

\begin{theorem}
In loopy stars $L(n)$, the winner alternates based on the parity of the total number of vertices. Specifically, Player 1 wins by two points in loopy stars with an odd number of branches, while Player 2 wins by three points in loopy stars with an even number of branches.
\end{theorem}

\begin{proof}
We break this situation into four cases, where the loopy star has 1 branch, 2 branches, three branches, or $n$ branches.

\textbf{Case 1:} Consider a loopy star $L(1)$, where the top vertex is connected to a single branch terminating in a loop.

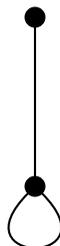
\begin{center}
    \begin{tikzpicture}[thick, main/.style={draw, circle, inner sep=0, outer sep=0, minimum size=2.5mm, fill=black}]
        \node[main] (1) {};
        \node[main] (2)[position=0:{20mm} from 1] {};
        \draw (2) -- (1);
        \draw (2) to [out=45, in=-45, looseness=20] (2);
    \end{tikzpicture}
    \captionof{figure}{The graph $L(1)$}
\end{center}

In this configuration, Player 1 can win by taking the edge and then the loop, resulting in a score of (2,0).

\textbf{Case 2:} A loopy star $L(2)$, where two branches each terminate in a loop.

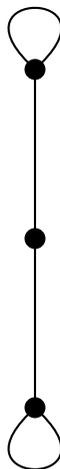
\begin{center}
    \begin{tikzpicture}[thick, main/.style={draw, circle, inner sep=0, outer sep=0, minimum size=2.5mm, fill=black}]
        \node[main] (1) {};
        \node[main] (2)[position=0:{20mm} from 1] {};
        \node[main] (3)[position=180:{20mm} from 1] {};
        \draw (2) -- (1) -- (3);
        \draw (2) to [out=45, in=-45, looseness=20] (2);
        \draw (3) to [out=225, in=135, looseness=20] (3);
    \end{tikzpicture}
    \captionof{figure}{The graph $L(2)$}
\end{center}

Player 2 wins this game with a score of (0, 3) since no matter what Player 1 does, it makes all of the vertices available for Player 2 to take.

\textbf{Case 3:} A loopy star $L(3)$, with three branches each terminating in a loop.

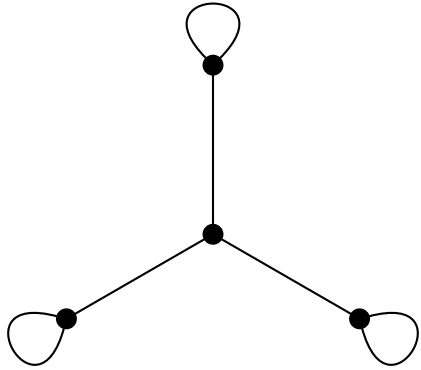
\begin{center}
    \begin{tikzpicture}[thick, main/.style={draw, circle, inner sep=0, outer sep=0, minimum size=2.5mm, fill=black}]
        \node[main] (1) {};
        \node[main] (2)[position=90:{20mm} from 1] {};
        \node[main] (3)[position=210:{20mm} from 1] {};
        \node[main] (4)[position=330:{20mm} from 1] {};
        \draw (2) -- (1) -- (3);
        \draw (1) -- (4);
        \draw (2) to [out=135, in=45, looseness=20] (2);
        \draw (3) to [out=255, in=165, looseness=20] (3);
        \draw (4) to [out=15, in=-75, looseness=20] (4);
    \end{tikzpicture}
    \captionof{figure}{The graph $L(3)$}
\end{center}

Player 1 wins with a score of (3, 1) since Player 1's first move will leave a vertex vulnerable. Player 2 can then either take that vertex and become Player 1 on the previous loopy star, which was a Player 1 loss with a score of (0, 3), resulting in a final score of (3, 1), or leave the vertex, resulting in a final score of (4, 0).

\textbf{Generalization:} Let $L(n)$ denote a loopy star graph with $n$ branches, where each branch consists of one edge connecting the central vertex to an outer vertex with one loop. Thus, $L(n)$ has $n$ loops, $n$ edges, and $n+1$ vertices total.

Consider the game on $L(n)$. Player 1 has two types of moves available: removing a loop or removing an edge. If Player 1 removes a loop, this leaves the corresponding outer vertex vulnerable, which Player 2 will capture, earning 1 point. After this move, the graph is reduced to $L(n-1)$ with Player 2 now acting as the first player on this smaller graph.

If Player 1 removes an edge, this isolates both the central vertex and one outer vertex with its loop. Player 2 will take this outer vertex and its loop, earning 2 points. The remaining graph is $L(n-1)$ but with the central vertex already taken, effectively making it equivalent to $n-1$ disconnected vertices each with a loop. In this scenario, Player 1 would need to take a loop, exposing a vertex for Player 2, and alternating until all vertices are taken. This results in a worse outcome for Player 1 than the first option.

Therefore, Player 1's optimal strategy is to remove a loop, after which Player 2 earns 1 point and effectively becomes Player 1 on $L(n-1)$.

Let $(x, y)$ be the score of the game on $L(n-1)$. When playing on $L(n)$, with Player 2 becoming the first player on $L(n-1)$, the roles are reversed, and the score becomes $(y, x+1)$ where the "+1" accounts for the initial point Player 2 earned.

Starting with $L(1)$, which has a score of (2, 0) in favor of Player 1, we can follow the pattern:
\begin{align*}
L(1): (2, 0) &\rightarrow L(2): (0, 2+1) = (0, 3) \\
&\rightarrow L(3): (3, 0+1) = (3, 1) \\
&\rightarrow L(4): (1, 3+1) = (1, 4) \\
&\rightarrow L(5): (4, 1+1) = (4, 2)
\end{align*}

This pattern can be expressed as follows: For odd $n = 2k+1$, Player 1 wins with a score of $(k+2, k)$, and for even $n = 2k$, Player 2 wins with a score of $(k-1, k+2)$, where $k \geq 1$.

By induction, this establishes that Player 1 always wins by 2 points when $n$ is odd, and Player 2 always wins by 3 points when $n$ is even, thus proving the theorem.
\end{proof}

\begin{table}[hbt!]
\centering
\begin{tabular}{|c|c|c|}
\hline
\textbf{\# of Branches} & \textbf{Winner} & \textbf{Score} \\
\hline
1 & P1 & (2, 0) \\
2 & P2 & (0, 3) \\
3 & P1 & (3, 1) \\
4 & P2 & (1, 4) \\
5 & P1 & (4, 2) \\
6 & P2 & (2, 5) \\
7 & P1 & (5, 3) \\
8 & P2 & (3, 6) \\
9 & P1 & (6, 4) \\
10 & P2 & (4, 7) \\
11 & P1 & (7, 5) \\
12 & P2 & (5, 8) \\
\hline
\end{tabular}
\caption{Computed outcomes for Loopy Stars $L(n)$ based on the number of branches.}
\label{table:hanging-star-outcomes}
\end{table}
\FloatBarrier

\section{Double Loopy Stars}

Double loopy stars are a variation of loopy stars where each branch terminates in two loops instead of one.  In terms of the general notation, a double loopy star has l = 1 and are of the form $L(n,2)$.

\begin{theorem}
In double loopy star graph games, Player 2 wins all games except for the smallest graph, $L(1, 2)$, which results in a tie. Specifically, aside from $L(1, 2)$ and $L(2, 2)$, where Player 2 wins by one point, Player 2 wins by two points when $n$ is odd and by three points when $n$ is even.
\end{theorem}

\begin{proof}
We analyze the outcomes of games played on small double loopy stars to discern a pattern that can be generalized.

\textbf{Case 1:} Consider the smallest double loopy star, $L(1, 2)$, which consists of a central vertex with a single branch terminating in two loops.

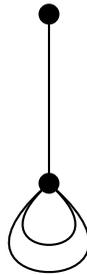
\begin{center}
    \begin{tikzpicture}[thick, main/.style={draw, circle, inner sep=0, outer sep=0, minimum size=2.5mm, fill=black}]
        \node[main] (1) {};
        \node[main] (2)[position=0:{20mm} from 1] {};
        \draw (2) -- (1);
        \draw (2) to [out=-45, in=45, looseness=20] (2);
        \draw (2) to [out=-47.5, in=47.5, looseness=30] (2);
    \end{tikzpicture}
\end{center}

This configuration results in a tie with a score of (1,1). Let us analyze all possible first moves for Player 1:

\begin{enumerate}
    \item If Player 1 takes the edge connecting the two vertices: Player 1 gains the central vertex (1 point). Player 2 then takes one loop from the outer vertex, forcing Player 1 to take the remaining loop and give up the outer vertex to Player 2. Final score: (1,1).
    
    \item If Player 1 takes one of the loops: Player 2 will take the remaining loop, forcing Player 1 to take the edge and give Player 2 the central vertex. Final score: (1,1).
\end{enumerate}

Thus, regardless of Player 1's opening move, the game ends in a tie.

\textbf{Case 2:} $L(2, 2)$, the next largest graph, has two branches each terminating in two loops.

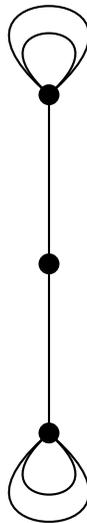
\begin{center}
    \begin{tikzpicture}[thick, main/.style={draw, circle, inner sep=0, outer sep=0, minimum size=2.5mm, fill=black}]
        \node[main] (1) {};
        \node[main] (2)[position=0:{20mm} from 1] {};
        \node[main] (3)[position=180:{20mm} from 1] {};
        \draw (2) -- (1) -- (3);
        \draw (2) to [out=-45, in=45, looseness=20] (2);
        \draw (2) to [out=-47.5, in=47.5, looseness=30] (2);
        \draw (3) to [out=225, in=135, looseness=20] (3);
        \draw (3) to [out=227.5, in=132.5, looseness=30] (3);
    \end{tikzpicture}
\end{center}

In $L(2, 2)$, Player 2 wins with a score of (1, 2). We need to consider all possible first moves for Player 1:

\begin{enumerate}
    \item If Player 1 takes one of the edges (optimal): Player 2 then claims the central vertex by taking the other edge and is forced to take a loop from one of the two remaining vertices. Player 1 then takes that vertex via the remaining loop and a loop from the other vertex, giving Player 2 the final point. Final score: (1,2).
    
    \item If Player 1 takes one of the loops: Player 2 will take the other loop from the same vertex, forcing Player 1 to take the edge connected to that vertex. This gives Player 2 both the central vertex and the remaining outer vertex. Final score: (0,3), which is worse for Player 1.
\end{enumerate}

Therefore, Player 1's optimal move is indeed to take one of the edges, limiting their loss to one point.

\textbf{Case 3:} $L(3, 2)$, with three branches each terminating in two loops.

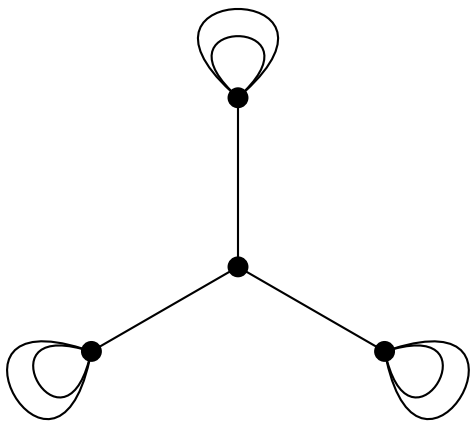
\begin{center}
    \begin{tikzpicture}[thick, main/.style={draw, circle, inner sep=0, outer sep=0, minimum size=2.5mm, fill=black}]
        \node[main] (1) {};
        \node[main] (2)[position=90:{20mm} from 1] {};
        \node[main] (3)[position=210:{20mm} from 1] {};
        \node[main] (4)[position=330:{20mm} from 1] {};
        \draw (2) -- (1) -- (3);
        \draw (1) -- (4);
        \draw (2) to [out=135, in=45, looseness=20] (2);
        \draw (2) to [out=137.5, in=42.5, looseness=30] (2);
        \draw (3) to [out=165, in=255, looseness=20] (3);
        \draw (3) to [out=162.5, in=257.5, looseness=30] (3);
        \draw (4) to [out=285, in=15, looseness=20] (4);
        \draw (4) to [out=282.5, in=17.5, looseness=30] (4);
    \end{tikzpicture}
\end{center}

In $L(3, 2)$, Player 2 wins with a score of (1, 3). Let's analyze all key first moves for Player 1:

\begin{enumerate}
    \item If Player 1 takes a loop: Player 2 takes an edge from a different vertex. If Player 1 then takes a loop from the vertex with two loops, Player 2 takes a loop from the vertex that is missing its edge. This allows Player 1 to take that vertex but then gives Player 2 all three remaining vertices. Final score: (1,3).
    
    \item If Player 1 takes an edge: Player 2 takes a loop from a different vertex. If Player 1 takes the loop from the vertex that already had a move played on it (the vertex missing an edge), Player 2 takes the remaining loop from that vertex. Player 1 then has no choice but to expose the central vertex and the remaining outer vertex, giving Player 2 all of them. Final score: (1,3).
    
    \item If Player 1 takes moves that lead to early isolation of a vertex: For example, if after Player 1 takes a loop and Player 2 takes an edge from a different vertex, Player 1 then takes the loop from the disconnected vertex, Player 2 would gain that vertex (1 point), take the loop from the untouched vertex, and then Player 1 would have to give up all remaining vertices. Final score: (0,4), which is worse for Player 1.
\end{enumerate}

Therefore, no matter what move Player 1 makes first, Player 2 has a response that leads to a win with a score of at least (1,3).

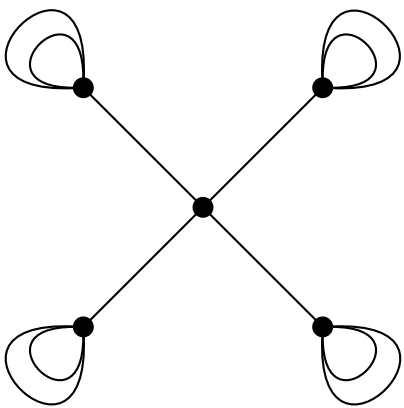
\textbf{Case 4:} Consider $L(4, 2)$, with four branches each terminating in two loops.

\begin{center}
    \begin{tikzpicture}[thick, main/.style={draw, circle, inner sep=0, outer sep=0, minimum size=2.5mm, fill=black}]
        \node[main] (1) {};
        \node[main] (2)[position=45:{20mm} from 1] {};
        \node[main] (3)[position=135:{20mm} from 1] {};
        \node[main] (4)[position=225:{20mm} from 1] {};
        \node[main] (5)[position=315:{20mm} from 1] {};
        \draw (2) -- (1) -- (3);
        \draw (5) -- (1) -- (4);
        \draw (2) to [out=90, in=0, looseness=20] (2);
        \draw (2) to [out=92.5, in=-2.5, looseness=30] (2);
        \draw (3) to [out=180, in=90, looseness=20] (3);
        \draw (3) to [out=182.5, in=87.5, looseness=30] (3);
        \draw (4) to [out=270, in=180, looseness=20] (4);
        \draw (4) to [out=272.5, in=177.5, looseness=30] (4);
        \draw (5) to [out=0, in=270, looseness=20] (5);
        \draw (5) to [out=2.5, in=267.5, looseness=30] (5);
    \end{tikzpicture}
\end{center}

$L(4,2)$ is where Player 2 can begin implementing their winning strategy. Let's consider the possible first moves:

\begin{enumerate}
    \item If Player 1 takes a loop: Player 2 takes an edge from a different vertex.
    \item If Player 1 takes an edge: Player 2 takes a loop from a different vertex.
\end{enumerate}

Whatever move Player 1 takes, Player 2 will do the opposite on a different vertex. Neither player will play on a vertex more than once during these initial moves, as doing so would give the other player that vertex. Additionally, if Player 1 gives up a vertex early, it would result in a worse outcome for them than following the main line of play.

After each player has taken two moves, $L(4,2)$ will have two vertices missing one edge and two vertices missing one loop. At this point, neither player will take from the vertices that have one loop and one edge, as doing so would mean giving up both of those vertices as well as the central vertex. Instead, the players will alternate back and forth, taking vertices until one of them is forced to give up the three remaining vertices (the central vertex and the last two outer vertices).

Before this alternation begins, four moves will have been taken, meaning that Player 1 will move first in the alternation phase and Player 2 will be the first to take a vertex. Player 1 then takes the next vertex and is ultimately forced to give up the three remaining vertices, leading to a Player 2 win with a score of (1,4).

\textbf{Generalization:} This strategy works for the general case $L(n,2)$. If Player 2 continues to take a loop when Player 1 takes an edge and takes an edge when Player 1 takes a loop, it will take $n$ moves to reach a graph where each outer vertex has been played on once. There are then $n - 2$ outer vertices that will be taken before someone has to give up the three remaining vertices (the central vertex and the last two outer vertices).

Since the players alternate turns taking vertices, there are $n + (n - 2) = 2n - 2$ decisions before someone has to give up the final three vertices. Since $2n - 2$ is always even, Player 1 will always be the one who gives up the final three vertices, resulting in a Player 2 win.

The final score follows a pattern: For odd $n \geq 3$, Player 2 wins by two points, and for even $n \geq 4$, Player 2 wins by three points, as confirmed by the computational results in Table 4.
\end{proof}

\begin{table}[hbt!]
\centering
\begin{tabular}{|c|c|c|}
\hline
\textbf{\# of Branches} & \textbf{Winner} & \textbf{Score} \\
\hline
1 & Tie & (1, 1) \\
2 & P2 & (1, 2) \\
3 & P2 & (1, 3) \\
4 & P2 & (1, 4) \\
5 & P2 & (2, 4) \\
6 & P2 & (2, 5) \\
7 & P2 & (3, 5) \\
8 & P2 & (3, 6) \\
9 & P2 & (4, 6) \\
10 & P2 & (4, 7) \\
11 & P2 & (5, 7) \\
12 & P2 & (5, 8) \\
\hline
\end{tabular}
\caption{Computed outcomes for Double Loopy Stars $L(n,2)$ based on the number of branches.}
\label{table:double-loopy-star-outcomes}
\end{table}
\FloatBarrier

\section{Loopy Starlike Trees}
Loopy starlike trees are an extension of loopy stars where the branches have more than one edge.  In terms of the general notation, loopy starlike trees are of the form $L(n,k,l)$ where $l > 1$. Note that $l$ refers to the number of edges in each branch. These generalized loopy trees have a single central vertex, and if the branches are sufficiently long (with 4 or more edges in a branch), Player 2 always wins. This is because once Player 1 cuts an edge along a branch, Player 2 can capture the majority of the branch, leaving one edge for Player 1 and subsequently forcing them to move onto another branch, where the strategy repeats.

\begin{observation}
For every loopy starlike tree $L(n, 1, l)$, with $l \geq 4$, Player 2 wins for any $n \geq 2$.
\end{observation}

\begin{proof}
Consider a loopy starlike tree $L(n, 1, l)$ with $n$ branches, each of length $l \geq 4$. Each branch consists of $l$ edges connecting $l+1$ vertices, with the last vertex having a loop attached.

When Player 1 cuts any edge on a branch, Player 2 employs the following strategy:
\begin{enumerate}
    \item Player 2 strategically cuts edges on that branch, capturing all but 2 vertices of the branch.
    \item Player 2 deliberately leaves one edge intact, which Player 1 must take.
    \item When Player 1 takes this edge, they gain the 2 remaining vertices on that branch.
    \item Player 1 is then forced to move to a new branch, and Player 2 repeats this strategy.
\end{enumerate}

Specifically, on each of the first $n-2$ branches, Player 2 captures $l-1$ vertices while Player 1 captures 2 vertices. This process continues until only two branches remain intact.

When Player 1 cuts an edge on one of these final two branches, Player 2 can capture all remaining vertices: the $l-1$ vertices on that branch, the entire last branch ($l+1$ vertices), and the central vertex. This gives Player 2 a total of $2l+1$ additional vertices at the end.

The final score is:
\begin{itemize}
    \item Player 1: $2(n-2)$ vertices (2 vertices from each of the first $n-2$ branches)
    \item Player 2: $(l-1)(n-2) + (2l+1)$ vertices ($(l-1)$ vertices from each of the first $n-2$ branches, plus $2l+1$ vertices at the end)
\end{itemize}

Since $l \geq 4$, we have $l-1 \geq 3$, which means Player 2 captures at least 3 vertices on each of the first $n-2$ branches while Player 1 captures only 2. This advantage, combined with Player 2's capture of all remaining vertices at the end, ensures Player 2 wins for any $n \geq 2$.
\end{proof}

\begin{figure}[hbt!]
\centering
\begin{tikzpicture}[main/.style = {draw, circle}]
    \node[main] (0) at (0,0) {};
    \foreach \angle/\name in {90/1, 162/2, 234/3, 306/4, 18/5}
    {
        \node[main] (\name A) at (\angle:1.5cm) {};
        \node[main] (\name B) at (\angle:2.5cm) {};
        \node[main] (\name C) at (\angle:3.5cm) {};
        \node[main] (\name D) at (\angle:4.5cm) {};
        \draw (0) -- (\name A) -- (\name B) -- (\name C) -- (\name D);
        \draw (\name D) to[out=\angle-40, in=\angle+40, looseness=10] (\name D);
    }
\end{tikzpicture}
\caption{An example of a loopy starlike tree $L(5, 1, 4)$.}
\label{fig:hanging-starlike-1}
\end{figure}
\FloatBarrier

\begin{observation}
For $L(n, 2, l)$, with $l \geq 4$, the winner alternates between Player 1 and Player 2. Player 1 wins if $n$ is odd, and Player 2 wins if $n$ is even.
\end{observation}

\begin{proof}
When a player cuts an edge on an untouched branch, the opponent can capture all edges in that branch, leaving behind a vertex with a single loop, thus forcing the next player to move onto another branch. As a result, players will initially avoid cutting edges and instead remove A loop from each branch. If the number of branches is odd, Player 1 can ensure that Player 2 is the first to move on an $L(n, 1, l)$ configuration, which leads to a loss for Player 2. Conversely, if the number of branches is even, Player 2 can force a win.
\end{proof}

\begin{observation}
For $L(n, 3, l)$ with $l \geq 4$, Player 2 wins for any $n$.
\end{observation}

\begin{proof}
Cutting an edge on a branch allows the opponent to capture all vertices in that branch, leaving behind a vertex with two loops. Therefore, Player 1's optimal initial move is to remove a loop. However, Player 2 will respond by taking the second loop on the same branch. This pattern continues on all branches until the graph is reduced to $L(n, 1, l)$, which results in a win for Player 2 for any number of branches $n$.
\end{proof}

\section{Complete Bipartite Graphs}

Complete bipartite graphs, denoted by $K(a, b)$, are graphs whose vertices can be divided into two disjoint sets such that every vertex in one set is connected to every vertex in the other set, and there are no edges within a set. The study of competitive graph games on complete bipartite graphs reveals interesting patterns based on the sizes of the partite sets. For simplicity, we will assume $a \leq b$ throughout this section.

\begin{theorem}
For all complete bipartite graphs $K(1, b)$ with $b \geq 1$, Player 1 wins.
\end{theorem}

\begin{proof}
In a complete bipartite graph $K(1, b)$, we have one vertex in the first partite set (call it the center vertex) connected to all $b$ vertices in the second partite set. This forms a star graph.

The game proceeds with no choices: when Player 1 takes an edge, they immediately gain one leaf vertex (since it becomes disconnected from the graph). Player 1 then takes another edge, gaining another leaf vertex, and continues this process until all $b$ leaf vertices are taken. Finally, Player 1 also gains the center vertex, for a total score of $(b+1, 0)$.
\end{proof}

\begin{figure}[hbt!]
\centering
\begin{tikzpicture}[main/.style = {draw, circle}]
    \node[main] (A) at (0, 0) {A};
    
    \foreach \y in {1,...,4}
    {
        \node[main] (B\y) at (2, 1.5-\y) {B\y};
        \draw (A) -- (B\y);
    }
\end{tikzpicture}
\caption{An example of a complete bipartite graph $K(1, 4)$.}
\label{fig:k1n}
\end{figure}
\FloatBarrier

\begin{theorem}
For all complete bipartite graphs $K(2, b)$ with $b \geq 2$, Player 1 wins when $b$ is odd and Player 2 wins when $b$ is even. Specifically, if $b = 2k$ is even, then the score is $(k-1, k+3)$ so Player 2 wins by 4 points. If $b = 2k+1$ is odd, then the score is $(k+3, k)$, so Player 1 wins by 3 points.
\end{theorem}

\begin{proof}
We will prove this theorem by induction on $b$.

\textbf{Base cases:}
\begin{itemize}
    \item $K(2, 2)$: This graph is isomorphic to a cycle $C_4$. Player 2 wins with a score of $(0, 4)$, which matches our formula with $k = 1$: $(k-1, k+3) = (0, 4)$.
    
    \item $K(2, 3)$: After Player 1's first move, one vertex in the second partite set will have degree 1. Player 2 must take the edge connecting to this vertex, gaining the vertex. The remaining graph is $K(2, 2)$, which is a Player 2 win with score $(0, 4)$. Thus, the overall score for $K(2, 3)$ is $(3, 1) + (0, 4) = (3, 1)$, which matches our formula with $k = 1$: $(k+3, k) = (4, 1)$.
\end{itemize}

\textbf{Inductive step:} Assume that for some $n \geq 3$, the theorem holds for all values of $b$ such that $2 \leq b < n$. We want to show it holds for $b = n$.

Consider the game on $K(2, n)$. When Player 1 makes the first move, they must take an edge without gaining any vertex (since all vertices in the second partite set have degree 2 initially). This leaves one vertex in the second partite set with degree 1. Player 2 then takes the edge connecting to this vertex, gaining the vertex.

The remaining graph is $K(2, n-1)$. By the induction hypothesis:

\begin{itemize}
    \item If $n-1 = 2k$ is even, then the score on $K(2, n-1)$ is $(k-1, k+3)$. Player 2 has already gained one vertex, so the overall score is $(k-1, k+3+1) = (k-1, k+4)$. Substituting $n = 2k+1$, we get a score of $(k-1, k+4) = ((n-1)/2 - 1, (n-1)/2 + 4) = ((n-3)/2, (n+7)/2)$. For $n = 2k+1$, this becomes $(k-1, k+4) = (k-1, (2k+1+7)/2) = (k-1, k+4)$, which matches our formula for odd $b$.
    
    \item If $n-1 = 2k+1$ is odd, then the score on $K(2, n-1)$ is $(k+3, k)$. Player 2 has already gained one vertex, so the overall score is $(k+3, k+1)$. Substituting $n = 2k+2$, we get a score of $(k+3, k+1) = (n/2 + 2, n/2 - 1)$. For $n = 2k+2$, this becomes $(k+3, k+1) = ((2k+2)/2 + 2, (2k+2)/2 - 1) = (k+3, k+1) = ((n-2)/2 + 3, (n-2)/2 - 1 + 2) = ((n+4)/2, (n+2)/2 - 1)$, which matches our formula for even $b$.
\end{itemize}

Therefore, by the principle of mathematical induction, the theorem holds for all $b \geq 2$.
\end{proof}

\begin{figure}[hbt!]
\centering
\begin{tikzpicture}[main/.style = {draw, circle}]
    \node[main] (A1) at (0, 0.5) {A1};
    \node[main] (A2) at (0, -0.5) {A2};
    
    \foreach \y in {1,...,3}
    {
        \node[main] (B\y) at (2, 1-\y) {B\y};
        \draw (A1) -- (B\y);
        \draw (A2) -- (B\y);
    }
\end{tikzpicture}
\caption{An example of a complete bipartite graph $K(2, 3)$.}
\label{fig:k2n}
\end{figure}
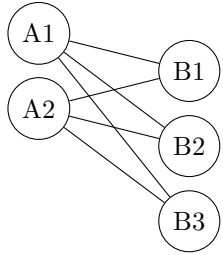
\FloatBarrier

\section{Other Graph Families}

In this section, we present results about graph families where we do not have a complete description of their behavior. 

\subsection{Loopy Cycles}

Loopy cycles, denoted as $C(n,k)$, represent a graph where a large cycle $C_n$ has $k$ loops added to it. These loops are attached to consecutive vertices on the cycle. Our computational results suggest an interesting pattern in the winning player based on the number of loops.

\begin{conjecture}
In loopy cycles $C(n,k)$ with $n \geq 3$ and loops added to consecutive vertices:
\begin{enumerate}
    \item For $k = 1$ and $k = 2$, Player 1 wins.
    \item For $k \geq 3$, the winner alternates between Player 2 (when $k$ is odd) and Player 1 (when $k$ is even).
\end{enumerate}
\end{conjecture}

For $C(n,1)$, Player 1 wins by taking the loop, thereby forcing Player 2 to break the cycle. This pattern of victory for Player 1 continues for $C(n,2)$, where Player 1 can secure a win by taking the edge connecting the two vertices with loops attached.

The scenario shifts with $C(n,3)$, favoring Player 2. Based on our computational results, we observe that for $k \geq 3$, Player 2 wins when $k$ is odd, and Player 1 wins when $k$ is even. This pattern of alternating winners for $k \geq 3$ provides an interesting strategic consideration in loopy cycle gameplay.

Here is an example of a loopy cycle $C(5,3)$:

\begin{figure}[hbt!]
\centering
\begin{tikzpicture}[thick, main/.style={draw, circle, inner sep=0, outer sep=0, minimum size=1.5mm, fill=black}]
    \node[main] (1) {};
    \node[main] (2)[position=-36:{15mm} from 1] {};
    \node[main] (3)[position=-108:{15mm} from 2] {};
    \node[main] (4)[position=-180:{15mm} from 3] {};
    \node[main] (5)[position=-252:{15mm} from 4] {};
    \draw (1) -- (2) -- (3) -- (4) -- (5) -- (1);
    \draw (1) to [out=135, in=45, looseness=25] (1);
    \draw (2) to [out=63, in=-27, looseness=25] (2);
    \draw (3) to [out=-9, in=-99, looseness=25] (3);
\end{tikzpicture}
\caption{An example of a loopy cycle $C(5,3)$ with three loops added to consecutive vertices.}
\label{fig:loopy-cycle}
\end{figure}
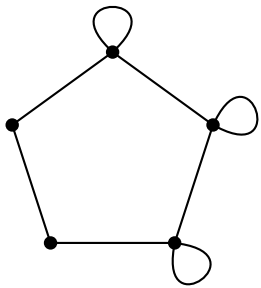
\FloatBarrier

Since $C(n,0)$ is just a cycle graph, it is a win for Player 2. However, as loops are added, the winning player alternates under certain conditions as described above.

\subsection{Wheel Graphs}

Wheel graphs, denoted as $W_n$, consist of a cycle with $n$ vertices with a central vertex connected to each vertex on the cycle. The central vertex and its connections form what can be visualized as the "spokes" of a wheel.

\begin{figure}[hbt!]
\centering
\begin{tikzpicture}[thick, main/.style={draw, circle, inner sep=0, outer sep=0, minimum size=1.5mm, fill=black}]
    \node[main] (1) {};
    \node[main] (2)[position=90:{15mm} from 1] {};
    \node[main] (3)[position=18:{15mm} from 1] {};
    \node[main] (4)[position=-54:{15mm} from 1] {};
    \node[main] (5)[position=-126:{15mm} from 1] {};
    \node[main] (6)[position=-198:{15mm} from 1] {};
    \draw (1) -- (2) -- (3) -- (1) -- (4) -- (5) -- (1) -- (6) -- (2);
    \draw (3) -- (4);
    \draw (5) -- (6);
\end{tikzpicture}
\caption{An example of a wheel graph with a cycle of 5 vertices and a central vertex connected to each vertex on the cycle.}
\label{fig:wheel-graph}
\end{figure}
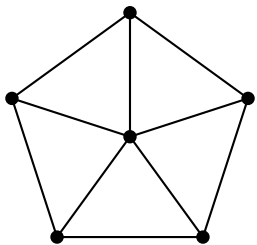
\FloatBarrier 

The table below summarizes the outcomes for wheel graphs with varying numbers of ``spokes,'' as calculated by the software. 

\begin{table}[hbt!]
\centering
\begin{tabular}{|c|c|c|}
\hline
\textbf{\# of Spokes} & \textbf{Winner} & \textbf{Score} \\
\hline
3 & P2 & (0, 4) \\
4 & P1 & (4, 1) \\
5 & P2 & (2, 4) \\
6 & P1 & (5, 2) \\
7 & P2 & (2, 6) \\
8 & P1 & (5, 4) \\
9 & P2 & (4, 6) \\
10 & P1 & (6, 5) \\
11 & P2 & (5, 7) \\
12 & P1 & (7, 6) \\
\hline
\end{tabular}
\caption{Computed outcomes for wheel graphs based on the number of spokes}
\label{table:wheel-graph-outcomes}
\end{table}
\FloatBarrier 

Based on our computational results, there appears to be a pattern where the winner alternates based on the parity of the number of spokes: Player 1 wins when the number of spokes is even (except for 4), and Player 2 wins when the number of spokes is odd.

Wheel graphs bear a conceptual relationship to loopy cycles. While in loopy cycles we attach loops to vertices on the cycle, in wheel graphs we connect those vertices to a central vertex. This structural correspondence might provide insights into strategies for both families of graphs.

\subsection{Ferris Wheel Graphs}

Ferris Wheel graphs, denoted as $C(n,n)$, are loopy cycles where every vertex on the cycle has a loop. They have a close relationship with wheel graphs $W_n$, as they can be viewed as transforming the wheel graph structure: instead of having a central vertex connected to each vertex on the cycle (as in a wheel graph), each vertex on the cycle has a loop.

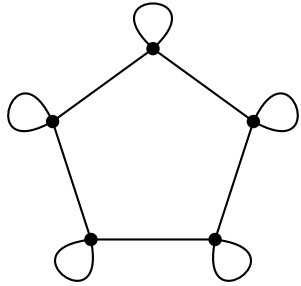
\begin{figure}[hbt!]
\centering
\begin{tikzpicture}[thick, main/.style={draw, circle, inner sep=0, outer sep=0, minimum size=1.5mm, fill=black}]
    \node[main] (1) {};
    \node[main] (2)[position=-36:{15mm} from 1] {};
    \node[main] (3)[position=-108:{15mm} from 2] {};
    \node[main] (4)[position=-180:{15mm} from 3] {};
    \node[main] (5)[position=-252:{15mm} from 4] {};
    \draw (1) -- (2) -- (3) -- (4) -- (5) -- (1);
    \draw (1) to [out=135, in=45, looseness=25] (1);
    \draw (2) to [out=63, in=-27, looseness=25] (2);
    \draw (3) to [out=-9, in=-99, looseness=25] (3);
    \draw (4) to [out=-171, in=-81, looseness=25] (4);
    \draw (5) to [out=-243, in=-153, looseness=25] (5);
\end{tikzpicture}
\caption{An illustration of a Ferris Wheel graph, showing a cycle with a central vertex connected to each vertex on the cycle and an additional edge to form a loop.}
\label{fig:ferris-wheel-graph}
\end{figure}
\FloatBarrier

Our computational results have revealed an interesting pattern between Ferris Wheel and wheel graph outcomes. The winner of a Ferris Wheel graph is the same as the winner of the corresponding wheel graph. Furthermore, the Ferris Wheel score is the same as the wheel graph score, but with one fewer point given to the winner.

\begin{conjecture}
For any $n \geq 3$, if Player X wins on the wheel graph $W_n$ with a score of $(a, b)$, then the Ferris Wheel graph $C(n,n)$ has a score of $(a-1, b)$ if Player X is Player 1, or $(a, b-1)$ if Player X is Player 2.
\end{conjecture}

This suggests that the additional point in the Wheel Graph score likely corresponds to the center vertex in the wheel graph. We observe that both graph families display a pattern change at $n=8$. If this conjecture holds, then understanding the winning strategies on Ferris Wheel graphs could provide insights into solving wheel graphs, and vice versa.



\begin{table}[hbt!]
\centering
\begin{tabular}{|c|c|c|}
\hline
\textbf{\# of Vertices} & \textbf{Winner} & \textbf{Score} \\
\hline
3 & P2 & (0, 3) \\
4 & P1 & (3, 1) \\
5 & P2 & (2, 3) \\
6 & P1 & (4, 2) \\
7 & P2 & (2, 5) \\
8 & Tie & (4, 4) \\
9 & P2 & (4, 5) \\
10 & Tie & (5, 5) \\
11 & P2 & (5, 6) \\
\hline
\end{tabular}
\caption{Computed outcomes for Ferris Wheel graphs based on the number of vertices.}
\label{table:balloon-cycle-outcomes}
\end{table}
\FloatBarrier

A formal proof of the relationship between wheel graphs and Ferris Wheel graphs remains open, but the computational evidence strongly suggests a structural correspondence in gameplay between these two families.

\subsection{Balloon Paths}

Balloon Paths, denoted as $BP_n$, are paths of length $n$ (with $n+1$ vertices) where each vertex has exactly one loop attached. These graphs extend the concept of path graphs by incorporating loops at vertices.

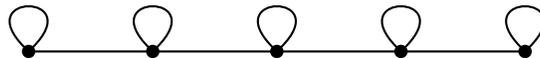
\begin{figure}[hbt!]
\centering
\begin{tikzpicture}[thick, main/.style={draw, circle, inner sep=0, outer sep=0, minimum size=1.5mm, fill=black}]
    \node[main] (1) {};
    \node[main] (2)[position=0:{15mm} from 1] {};
    \node[main] (3)[position=0:{15mm} from 2] {};
    \node[main] (4)[position=0:{15mm} from 3] {};
    \node[main] (5)[position=0:{15mm} from 4] {};
    \draw (1) -- (2) -- (3) -- (4) -- (5);
    \draw (1) to [out=135, in=45, looseness=25] (1);
    \draw (2) to [out=135, in=45, looseness=25] (2);
    \draw (3) to [out=135, in=45, looseness=25] (3);
    \draw (4) to [out=135, in=45, looseness=25] (4);
    \draw (5) to [out=135, in=45, looseness=25] (5);
\end{tikzpicture}
\caption{An illustration of a Balloon Path graph $BP_4$, showing a path of length 4 with a loop at each vertex.}
\label{fig:balloon-path-graph}
\end{figure}
\FloatBarrier

While a comprehensive proof is pending, our preliminary analysis based on induction and game simulations suggests a pattern tied to the parity of the path length in Balloon Paths.

For Balloon Paths of odd length ($BP_{2k+1}$), it appears that the game results in a tie, with players effectively canceling out each other's advantages. For Balloon Paths of even length ($BP_{2k}$), Player 1 consistently wins, with the winning margin alternating between one and three points. We can at least prove the following for even-length Balloon Paths:

\begin{proposition} When $n$ is odd, Player 1 may force a tie on $BP_n$. \label{Odd balloon paths are ties}
\end{proposition}
\begin{proof}
Player 1 is able to force a tie on an odd-length balloon path by taking the middle edge, creating two identical smaller balloon paths. Player 1 then mirrors Player 2's moves, where for each move Player 2 makes, Player 1 makes the same move on the other copy of that balloon path. The exception is that if Player 2 offers a vertex for Player 1 to take, Player 1 should take that vertex and then make the same offer to Player 2. Due to this symmetry strategy, Player 2 cannot win.

The only way this strategy could fail is if there exists some move that cannot be mirrored. However, since both subgraphs are identical and each vertex has exactly one loop and at most two incident edges, any move that Player 2 makes can always be mirrored by Player 1 in the corresponding subgraph. This ensures that for every vertex Player 2 captures, Player 1 captures a corresponding vertex.
\end{proof}

Note that the proof of Proposition \ref{Odd balloon paths are ties} is only very weakly using that the graph is a balloon path. We have the more general following proposition using the same argument. 

\begin{proposition} Let $G$ be a graph where every vertex has degree at least $2$, and let $H$ be a graph made from taking two copies of $G$ and connecting them by a single edge (not necessarily copied at the corresponding vertex). Then Player 1 may force a tie on $H$.
\end{proposition} 
\begin{proof} The proof is essentially the same argument as with Proposition \ref{Odd balloon paths are ties}. Player 1 takes the connecting edge, and then mirrors each move Player 2 makes, with the exception of when offered a vertex, they take it and then make the symmetric offer. Note that we are using that every vertex is of degree at least 2 to make sure that Player 2 cannot take any vertices on their first move.
\end{proof}

\begin{table}[hbt!]
\centering
\begin{tabular}{|c|c|c|}
\hline
\textbf{Path Length} & \textbf{Winner} & \textbf{Score} \\
\hline
2 & P1 & (3, 0) \\
3 & Tie & (2, 2) \\
4 & P1 & (3, 2) \\
5 & Tie & (3, 3) \\
6 & P1 & (5, 2) \\
7 & Tie & (4, 4) \\
8 & P1 & (5, 4) \\
9 & Tie & (5, 5) \\
10 & P1 & (6, 5) \\
11 & Tie & (6, 6) \\
12 & P1 & (7, 6) \\
\hline
\end{tabular}
\caption{Computed outcomes for Balloon Path graphs based on the length of the path.}
\label{table:balloon-path-outcomes}
\end{table}

\FloatBarrier 

\subsection{Complete Graphs}

Complete graphs, denoted as $K_n$, are characterized by each vertex being connected to every other vertex. One might hope that the analysis of complete graphs would be relatively straightforward due to their high level of symmetry. While there does seem to be an empirically simple pattern, proving this pattern has proven to be difficult.

Table \ref{table:complete-graph-outcomes} shows the data for $K_n$. From the data, it appears that $K_n$ is a Player 1 win when $n \equiv 1$ or $2 \pmod{4}$, and a Player 2 win when $n \equiv 0$ or $3 \pmod{4}$.

\begin{table}[hbt!]
\centering
\begin{tabular}{|c|c|c|}
\hline
\textbf{\# of Vertices} & \textbf{Winner} & \textbf{Score} \\
\hline
1 & P1 & (2, 0) \\
2 & P1 & (2, 0) \\
3 & P2 & (0, 3) \\
4 & P2 & (0, 4) \\
5 & P1 & (4, 1) \\
6 & P1 & (5, 1) \\
7 & P2 & (2, 5) \\
8 & P2 & (2, 6) \\
9 & P1 & (7, 2) \\
10 & P1 & (7, 3) \\
\hline
\end{tabular}
\caption{Computed outcomes for complete graphs based on the number of vertices and their congruence modulo 4.}
\label{table:complete-graph-outcomes}
\end{table}
\FloatBarrier 

More time and computing power would be needed to verify that this pattern persists beyond 10 vertices, as $K_{11}$ is estimated to take 220 days to calculate with the current algorithm.

We can at least provide a plausibility argument that $K_{4n}$ should be a Player 2 win. We split the vertices of $K_{4n}$ into four sets, $A$, $B$, $C$, and $D$, each with the same number of vertices, and with each vertex associated with exactly one other vertex in the other three subsets.

We visualize the four sets as splitting up $K_{4n}$ into quadrants, so we have the following picture.
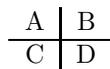
\begin{table}[htbp]
\centering
\begin{tabular}{c|c}
A & B \\
\hline
C & D  
\end{tabular}
\label{table:quadrant-visualization}
\end{table}
\FloatBarrier

Player 2 then mirrors Player 1 in the following way: If Player 1 makes a move with both vertices in a set, Player 2 makes the mirror image move in the corresponding diagonal set (so sets $A$ and $D$ are associated, and sets $B$ and $C$ are associated). If Player 1 makes a move with vertices in different sets, Player 2 makes the same move in the pair of other sets.

The exception is that if Player 1 makes an offer, Player 2 accepts the offer and then makes the corresponding offer. Unless there is a forest at the end, in which case Player 2 takes everything. This strategy seems to force a Player 2 win, as the game will remain tied until there is a forest, where Player 2 will win.

The problem with this argument is what happens if, when Player 2 accepts an offer and then makes an offer, Player 1 does not accept the offer? In this case, there will be a break in symmetry, and the mirroring strategy no longer works. However, this means that for Player 1 to break the mirroring, they will need to have given up at least one vertex, which should put them at a disadvantage.

\subsection{Petersen Graph}

The Petersen Graph is included in our analysis not as part of a systematic family study, but because of its significance as a standard test case in graph theory. Known for its unique properties and frequent appearance in counterexamples, the Petersen Graph serves as an interesting standalone case for our Dots and Boxes variant.

\begin{figure}[hbt!]
\centering
\begin{tikzpicture}[main/.style = {draw, circle}, rotate=18]
    \node[main] (1) at (0*72:2) {};
    \node[main] (2) at (1*72:2) {};
    \node[main] (3) at (2*72:2) {};
    \node[main] (4) at (3*72:2) {};
    \node[main] (5) at (4*72:2) {};
    \node[main] (6) at (0*72:1) {};
    \node[main] (7) at (1*72:1) {};
    \node[main] (8) at (2*72:1) {};
    \node[main] (9) at (3*72:1) {};
    \node[main] (10) at (4*72:1) {};

    \draw (1) -- (2) -- (3) -- (4) -- (5) -- (1);
    \draw (6) -- (8) -- (10) -- (7) -- (9) -- (6);
    \draw (1) -- (6);
    \draw (2) -- (7);
    \draw (3) -- (8);
    \draw (4) -- (9);
    \draw (5) -- (10);
\end{tikzpicture}
\caption{The Petersen Graph}
\label{fig:petersen-graph}
\end{figure}
\FloatBarrier

Our computational analysis reveals that on the Petersen Graph, Player 1 wins with a score of (9, 1). This decisive victory for Player 1 is noteworthy given the graph's structural complexity and symmetry, and adds to our understanding of how graph structure influences gameplay outcomes in this variant of Dots and Boxes.

\subsection{Hypercubes}

Hypercubes, also known as $n$-dimensional cubes or $n$-cubes, are a family of graphs that generalize the notion of a cube to higher dimensions.

Here is an illustration of a 3-dimensional hypercube, $Q_3$:

\begin{figure}[hbt!]
\centering
\begin{tikzpicture}[main/.style = {draw, circle}, scale=1.5]
    \node[main] (1) at (0,0) {};
    \node[main] (2) at (2,0) {};
    \node[main] (3) at (2,2) {};
    \node[main] (4) at (0,2) {};
    \node[main] (5) at (1,1) {};
    \node[main] (6) at (3,1) {};
    \node[main] (7) at (3,3) {};
    \node[main] (8) at (1,3) {};
    
    \draw (1) -- (2) -- (3) -- (4) -- (1);
    \draw (5) -- (6) -- (7) -- (8) -- (5);
    \draw (1) -- (5);
    \draw (2) -- (6);
    \draw (3) -- (7);
    \draw (4) -- (8);
\end{tikzpicture}
\caption{An illustration of a 3-dimensional hypercube, $Q_3$.}
\label{fig:hypercube}
\end{figure}
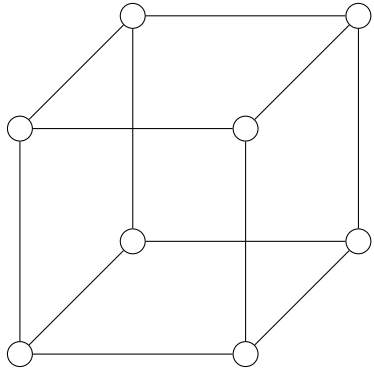
\FloatBarrier 

The table below summarizes the outcomes of games played on hypercubes with different dimensions. 

\begin{table}[hbt!]
\centering
\begin{tabular}{|c|c|c|}
\hline
\textbf{\# of Dimensions} & \textbf{Winner} & \textbf{Score} \\
\hline
1 & P1 & (2, 0) \\
2 & P2 & (0, 4) \\
3 & P2 & (2, 6) \\
4 & P2 & (6, 10) \\
\hline
\end{tabular}
\caption{Computed outcomes for Hypercubes based on the number of dimensions.}
\label{table:hypercubes-outcome}
\end{table}
\FloatBarrier 

\subsection{Prism graphs}

Prism graphs are a class of graphs formed by connecting two $n$-gons (regular polygons with $n$ sides) at corresponding vertices. Each vertex in one $n$-gon is connected to the corresponding vertex in the other $n$-gon, creating a 3-dimensional structure that can be visualized as a prism. Equivalently, the $n$th prism graph is obtained by taking the Cartesian product of the $n$th cycle graph with $P_2$. We denote the $n$th prism graph, made by connecting two $n$-gons, as $PG_n$.

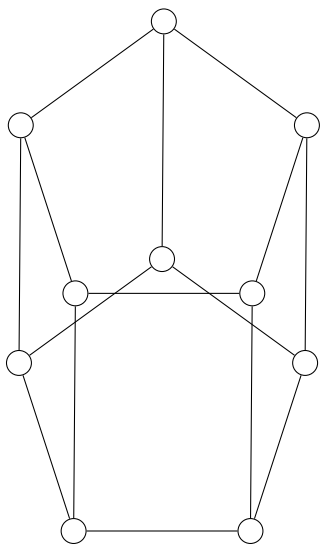
\begin{figure}[hbt!]
\centering
\begin{tikzpicture}[main/.style = {draw, circle}, rotate=18, scale=2]
    \node[main] (1) at (0:1) {};
    \node[main] (2) at (72:1) {};
    \node[main] (3) at (2*72:1) {};
    \node[main] (4) at (3*72:1) {};
    \node[main] (5) at (4*72:1) {};
    
    \node[main] (6) at ($(0:1)+(0.5,1.5)$) {};
    \node[main] (7) at ($(72:1)+(0.5,1.5)$) {};
    \node[main] (8) at ($(2*72:1)+(0.5,1.5)$) {};
    \node[main] (9) at ($(3*72:1)+(0.5,1.5)$) {};
    \node[main] (10) at ($(4*72:1)+(0.5,1.5)$) {};
    
    \draw (1) -- (2) -- (3) -- (4) -- (5) -- (1);
    \draw (6) -- (7) -- (8) -- (9) -- (10) -- (6);
    \draw (1) -- (6);
    \draw (2) -- (7);
    \draw (3) -- (8);
    \draw (4) -- (9);
    \draw (5) -- (10);
\end{tikzpicture}
\caption{An illustration of a prism graph, showing two 5-gons connected at corresponding vertices.}
\label{fig:double-ngon}
\end{figure}
\FloatBarrier 

Table \ref{table:double-ngon-outcomes} presents the outcomes for prism graphs for various $n$.

\begin{table}[hbt!]
\centering
\begin{tabular}{|c|c|c|}
\hline
\textbf{Size of $n$-gon} & \textbf{Winner} & \textbf{Score} \\
\hline
3 & P1 & (5, 1) \\
4 & P2 & (2, 6) \\
5 & P1 & (8, 2) \\
6 & P2 & (4, 8) \\
7 & P1 & (8, 6) \\
8 & P2 & (7, 9) \\
9 & Tie & (9, 9) \\
10 & P2 & (9, 11) \\
\hline
\end{tabular}
\caption{Computed outcomes for prism graphs based on $n$.}
\label{table:double-ngon-outcomes}
\end{table}
\FloatBarrier 

There is an apparent pattern where the winner alternates based on whether $n$ is even or odd, but the pattern breaks down for $n=9$ which is a tie. 


    

\vskip 20pt\noindent {\bf Acknowledgement.}
This paper was written as part of the Hopkins School Mathematics Seminar 2023-2024. 

\bibliographystyle{unsrt}
\bibliography{references}

\begin{thebibliography}{1}

\bibitem{Berlekamp}
E.~R. Berlekamp.
\newblock {\em The Dots-and-Boxes Game: Sophisticated Child's Play}.
\newblock A K Peters/CRC Press, 2000.

\bibitem{DD}
E.~D. Demaine and Y.~Diomidov.
\newblock Strings-and-coins and nimstring are pspace-complete.
\newblock {\em Integers}, 21B(A7), 2021.
\newblock (The John Conway, Richard Guy, and Elwyn Berlekamp Memorial Volume).

\bibitem{BHKM}
Kevin Buchin, Mart Hagedoorn, Irina Kostitsyna, and Max van Mulken.
\newblock Dots {\&} boxes is pspace-complete.
\newblock {\em CoRR}, abs/2105.02837, 2021.

\bibitem{JSWW}
A.~S. Jobson, L.~Sledd, S.~C. White, and D.~J. Wildstrom.
\newblock Variations on narrow dots-and-boxes and dots-and-triangles.
\newblock {\em Integers}, 17(G2), 2017.

\bibitem{Nowakowski}
R.~Nowakowski, editor.
\newblock {\em More Games of No Chance}, volume~42 of {\em Math. Sci. Res. Inst. Publ.}
\newblock Cambridge Univ. Press, Cambridge, 2002.

\bibitem{GitHub}
Modified dots and boxes.
\newblock \url{https://github.com/0xCUB3/Modified-Dots-and-Boxes}.

\end{thebibliography}

\end{document}